\documentclass[english]{amsart}
\usepackage[T1]{fontenc}
\usepackage[latin9]{inputenc}
\usepackage{geometry}
\usepackage{amsthm}
\usepackage{amssymb}
\usepackage{tikz-cd}
\usepackage{theoremref}
\usepackage{bbm}
\usepackage{color}

\makeatletter
\numberwithin{equation}{section}
\numberwithin{figure}{section}
\theoremstyle{plain}
\newtheorem{thm}{\protect\theoremname}
\theoremstyle{plain}
\newtheorem{prop}[thm]{\protect\propositionname}

\newtheorem{nota}{\protect\notationname}
\theoremstyle{definition}
\newtheorem{ex}[thm]{\protect\exname}

\makeatother

\usepackage{babel}
\providecommand{\propositionname}{Proposition}
\providecommand{\theoremname}{Theorem}
\providecommand{\corollaryname}{Corollary}
\providecommand{\notationname}{Notation}
\providecommand{\exname}{Example}

\begin{document}
\title[$\mathbb{A}^1$-contractible affine varieties of higher dimension]{Algebraic families of higher dimensional  $\mathbb{A}^{1}$-contractible affine varieties non-isomorphic to affine spaces} 
\author{Adrien Dubouloz}
\address{Laboratoire de Math\'ematique et Applications, UMR 7348 CNRS, Universit\'e de Poitiers, 86000 Poitiers, France. \vspace{-1em}}
\address{Institut de Math\'ematiques de Bourgogne, UMR 5584 CNRS, Universit\'e de Bourgogne, 21000 Dijon, France.}
\email{adrien.dubouloz@math.cnrs.fr}
\author{Parnashree Ghosh}
\address{School of Mathematics, Tata Institute of Fundamental Research, Mumbai 400005, India.}
\email{pmaths@math.tifr.res.in, ghoshparnashree@gmail.com}

\begin{abstract}
We construct algebraic families of smooth affine $\mathbb{A}^1$-contractible varieties of every dimension $n\geq 4$ over fields of characteristic zero which are non-isomorphic to affine spaces and potential counterexamples to the Zariski Cancellation Problem. We further prove that these families of varieties are also counter examples to the generalized Cancellation problem. 
\end{abstract}
\maketitle

\section*{Introduction}
The characterization of affine spaces $\mathbb{A}^{n}$ among algebraic varieties remains a central and still largely open problem in affine algebraic geometry, a source of numerous developments and contributions over several decades, involving complementary techniques from commutative algebra, geometry, and topology. The $\mathbb{A}^{1}$-homotopy theory of schemes has shed new light on this problem and raised new questions, in particular strengthening the natural analogy with the characterization of Euclidean affine spaces $\mathbb{R}^n$ in geometric topology \cite{AsokOstvaer22}. It is a classical result that for $n\leq 2$, Euclidean spaces $\mathbb{R}^n$ are the only contractible open topological manifolds up to homeomorphism. On the other hand, in all dimension $n\geq 3$, there exist a continuum of pairwise non-homeomorphic open contractible manifolds of dimension n. Moreover, all such manifolds have the property that their cartesian product with $\mathbb{R}$ become homeomorphic to $\mathbb{R}^{n+1}$. 

\smallskip

In the algebraic setting, the natural analogous problem is to understand the geometry of smooth algebraic varieties X over a field which are $\mathbb{A}^{1}$-contractible, that is, for which the structure morphism $X\to\mathrm{Spec}(k)$ is an $\mathbb{A}^{1}$-weak equivalence in the $\mathbb{A}^{1}$-homotopy homotopy category $\mathbf{H}(k)$ of $k$-schemes of Morel and Voevodsky \cite{MV}. The question is intimately linked to the Zariski Cancellation Problem \cite{GuCan,RuCan} in particular for the fact that a smooth $k$-variety $X$ of dimension n such that $X\times_{k}\mathbb{A}_{k}^{1}$ is isomorphic to $\mathbb{A}_{k}^{n+1}$ is necessarily $\mathbb{A}^{1}$-contractible.
Asok and Doran \cite[Theorem 5.1]{AD} discovered a first infinitely countable collections of strictly quasi-affine $\mathbb{A}^{1}$-contractible varieties of any dimension $n\geq 4$. Over fields of positive characteristic, infinite families of examples of $\mathbb{A}^{1}$-contractible smooth affine varieties of every dimension $n\geq 3$ non-isomorphic to affine spaces, called generalized Asanuma varieties, were found as by-products of the negative solution to the Zariski Cancellation Problem in positive characteristic due Gupta \cite{G,G2,GG}.
Later on, examples of smooth affine $\mathbb{A}^{1}$-contractible varieties $X$ of dimension 3 over fields of characteristic zero non-isomorphic to $\mathbb{A}_{k}^{3}$ and admitting a positive dimensional moduli were constructed by Fasel and the first author \cite{DF}. Affine cylinders $X\times_{k}\mathbb{A}_{k}^{m}$, $m\geq 1$ over these varieties provide natural collections of $\mathbb{A}^{1}$-contractible affine varieties of any dimension $n\geq 4$, but it is not known so far whether any of these varieties is isomorphic to an affine space or not, in particular they all remain potential counterexamples to the Zariski Cancellation Problem in characteristic zero. 

On the other hand, it is known by classification that the affine line $\mathbb{A}_{k}^{1}$ is the only $\mathbb{A}^{1}$-contractible smooth curve over any base field and very recent results of Choudhury and Roy \cite{CR} confirmed that over a field of characteristic zero, the affine plane $\mathbb{A}_{k}^{2}$ is the only $\mathbb{A}^{1}$-contractible smooth affine surface up to isomorphism. Thus, as far as the existence problem for "exotic" smooth affine $\mathbb{A}^{1}$-contractible varieties is concerned, the remaining  questions are whether there exist smooth affine $\mathbb{A}^{1}$-contractible surfaces non-isomorphic to $\mathbb{A}_{k}^{2}$ over some fields of positive characteristic and whether there exist smooth affine $\mathbb{A}^{1}$-contractible varieties of dimension $n\geq 4$ non-isomorphic to $\mathbb{A}_{k}^{n}$ over some fields of characteristic zero.

\smallskip

 In this article, we settle the second question by constructing for every field $k$ of characteristic zero and every $n\geq 4$ algebraic families of pairwise non-isomorphic $\mathbb{A}^1$-contractible affine $k$-varieties of dimension $n$. These arise as natural higher dimensional generalizations of the deformed Koras-Russell threefolds considered in \cite{DF,DMP} and besides their $\mathbb{A}^1$-contractibility, they share with them many other algebro-geometric properties in relation to the Zariski Cancellation Problem and its generalizations. The following theorem offers a concise illustration of the main results of the article in a specific case.
 
\medskip

\noindent {\bf Theorem.}
Let $k$ be a field of characteristic zero,  $n\geq 4$ be an integer and let $\mathcal{X}$ be the affine variety in $\mathbb{A}^{n-3}_k\times_k \mathbb{A}^{m+4}_k=\mathrm{Spec}(k[a_2,\ldots, a_{n-2}][x_0,\ldots, x_{m},y,z,t])$ with equation 
$$
\underline{x}^ny+z^2+t^3+x_0(1+\underline{x} +\sum_{i=2}^{n-2} a_i \underline{x}^i)=0,
$$ where $\underline{x}=\prod_{i=0}^m x_i$. Then, viewing $\mathcal{X}$ as a scheme   $f=\mathrm{pr}_1:\mathcal{X}\to \mathbb{A}^{n-3}_k$, the following hold:
\begin{itemize}
    \item [1.] The fibers of $f:\mathcal{X}\to \mathbb{A}^{n-3}_k$ over points of $S$ (closed or not) are $\mathbb{A}^1$-contractible smooth affine varieties of dimension $m+3$ non-isomorphic to $\mathbb{A}^{m+3}$ over the corresponding residue fields. 

    \item[2.]  The fibers $\mathcal{X}_s$ of $f:\mathcal{X}\to \mathbb{A}^{n-3}_k$  over $k$-rational points $s\in \mathbb{A}^{n-3}_k(k)$ are pairwise non-isomorphic $k$-varieties whose $\mathbb{A}^1$-cylinders $\mathcal{X}_s\times_k\mathbb{A}^1_k$ are all isomorphic.
\end{itemize}

The second assertion of the theorem states in particular that the $k$-varieties $\mathcal{X}_s$ add to the  list of counterexamples to the generalized Cancellation Problem. Moreover, we show that the invariants known to date do not allow to distinguish the $k$-varieties $\mathcal{X}_s\times_k \mathbb{A}^1_k$ from an affine space, making all the varieties $\mathcal{X}_s$ potential counterexamples to the Zariski Cancellation Problem itself.

\medskip

The scheme of the article is the following: section one is devoted to 
establishing
basic properties of a class of affine varieties containing those of Theorem A. The techniques employed 
involve mainly algebraic invariants of actions of the additive group $\mathbb{G}_{a,k}$ on affine $k$-varieties  defined through their associated locally nilpotent $k$-derivations, 
see \cite{FreuBook}  for an account on this theory. 
The $\mathbb{A}^1$-contractibility of varieties in this class is established in section two
by geometric techniques involving an induction on their dimension. 

\medskip

\textit{Acknowledgements.} The present research was done in part during the stay of the authors at the CIRM in November 2024 at the occasion of the international conference "Motivic Homotopy in Interaction". The authors are grateful to Anand Sawant for fruitful discussions.  
The first author is supported in part by the ANR Grant "HQ-Diag" ANR-21-CE40-0015.

\section{Isomorphism types and $\mathbb{A}^1$-cylinders}

In what follows, $k$ denotes a field of characteristic zero. For any ring $R$ and positive integer $n$, $R^{[n]}$ denotes a polynomial ring in $n$-variables over $R$.

\begin{nota}\label{mainnota} We fix integers $q,r\geq2$ with $\gcd(q,r)=1$ and for every triple $(m,\underline{n}, p)$ where $m\geq 0$ is an integer, $\underline{n}=(n_0,\ldots, n_m)$ is an $(m+1)$-tuple of integers $n_i>1$ and $p \in k[u]$ is a polynomial with $p(0) \neq 0$, we put 
$$R_m(\underline{n},p)=k[x_0,\ldots,x_m,y,z,t]/\left(\underline{x}^{\underline{n}}y +z^q+t^r+x_0p(\underline{x})\right),$$
where $\underline{x}^{\underline{n}}=\prod_{i=0}^{m}x_{i}^{n_{i}}$ and $\underline{x}=\prod_{i=0}^m x_i$. We let $X_m(\underline{n},p)=\mathrm{Spec}\, (R_m(\underline{n},p))$. 
\end{nota}

Since $p(0)\neq 0$, it follows readily from the Jacobian criterion that every scheme $X_m(\underline{n},p)$ is a non-singular $k$-variety, of dimension $m+3$. 

\subsection{Isomorphism types}
We first review of basic properties concerning the isomorphism types of the $k$-algebras $R_m(\underline{n},p)$ and their associated $k$-varieties $X_m(\underline{n},p)$. We begin with the following proposition which asserts in particular that none of the $k$-varieties $X_m(\underline{n},p)$ is isomorphic the affine space $\mathbb{A}_k^{m+3}$.

\begin{prop}\thlabel{b} For every triple $(m,\underline{n},p)$, the Makar-Limanov and Derksen invariants of the $k$-algebra $R_m(\underline{n},p)$ are respectively equal to the subalgebras
$$\mathrm{ML}(R_m(\underline{n},p))=k[x_0,\ldots,x_m] \quad \textrm{and} \quad  \mathrm{DK}(R_m(\underline{n},p))=k[x_0,\ldots,x_m,z,t].$$ 
In particular, none of the $k$-algebras $R_m(\underline{n},p)$ is isomorphic to $k^{[m+3]}$.
\end{prop}
\begin{proof}
Since $k[z,t]/(z^q+t^r)$ is not normal, it follows from \cite[Proposition 3.4(i)]{G2} and \cite[Proposition 2.10]{GG}, that $\mathrm{DK}(R_m(\underline{n},p))=k[x_0,\ldots,x_m,z,t]$ and from \cite[Proposition 3.4]{GG} that $\mathrm{ML}(R_m(\underline{n},p))=k[x_0,\ldots,x_m]$. Since $\mathrm{ML}(R_m(\underline{n},p))\neq k$,  $R_m(\underline{n},p)$ is not isomorphic to a polynomial ring over $k$.
\end{proof}

The following proposition provides in turn a partial classification of the isomorphism classes of varieties of the form $X_m(\underline{n},p)$ for a fixed $(m+1)$-tuple $\underline{n}=(n_0,\ldots, n_m)$. 
\begin{prop}\thlabel{c} For polynomials  $p_1,p_2 \in k[u]$ of degree $\leqslant n_0-2$, the $k$-algebras $R_m(\underline{n},{p_1})$ and  $R_m(\underline{n},{p_2})$ are isomorphic if and only if there exist $\mu, \lambda \in k^*$ such that $p_2(u)= \mu p_1(\lambda u)$. 
\end{prop}
\begin{proof}
Put $D=k[x_{0},\ldots, x_{m}, z,t]$. The inclusion $\sigma_{m,\underline{n},p}:\mathrm{DK}(R_m(\underline{n},p))=D\to R_m(\underline{n},p)$ expresses  $R_m(\underline{n},p)$ as the affine blow-up algebra 
$$D[\frac {J_m(\underline{n},p)}{\underline{x}^{\underline{n}}}]\subset D[\underline{x}^{-1}]$$ of $D$, where $J_m(\underline{n},p)$ is the ideal generated by $\underline{x}^{\underline{n}}$ and  $z^{q}+t^{r}+x_{0}p(\underline{x})$.
Put  $R_i=R_m(\underline{n},{p_i})$ and $\sigma_i=\sigma_{m,\underline{n},p_i}$.
An isomorphism of $k$-algebras $\phi:R_1\to R_2$ restricting to isomorphisms between the respective Derksen and Makar-Limanov invariants of $R_1$ and $R_2$, it follows that we have a commutative diagram 
$$
\begin{tikzcd}
k[x_{0},\ldots, x_{m}] \arrow[r] \arrow[d, "\phi_{\mathrm{ML}}"]  & k[x_{0},\ldots, x_{m}, z,t]  \arrow[r,"\sigma_1"] \arrow[d, "\phi_{\mathrm{DK}}"] & R_1\arrow[d, "\phi"]\\
k[x_{0},\ldots, x_{m}]\arrow[r] & k[x_{0},\ldots, x_{m}, z,t]  \arrow[r,"\sigma_2"] & R_2
\end{tikzcd}
$$
in which the vertical maps are $k$-algebra isomorphisms. Now by \cite[Theorem 4.1]{GG}, it follows that for every $i\in \{0,\ldots, m\}$, $\phi(x_i)=\lambda_{i}x_{j(i)}$ for some $j(i)\in \{0,\ldots, m\}$ such that $n_{j(i)}=n_i$ and some $\lambda_{i}\in k^*$.  
Noting that $R_i/(x_{j})$ is non-normal if $j=0$ and isomorphic to a polynomial over $k$ otherwise, we conclude that $\phi(x_{0})=\lambda_0x_{0}$. Put $\nu_0=(\prod_{i=1}^m\lambda_i)^{-1}$ and let $\alpha$ be the $k$-algebra automorphism of $D$ defined by $$\alpha(x_0,x_1,\ldots, x_i, \ldots, x_m,z,t)=(\nu_0x_0, \lambda_1 x_{j(1)}, \ldots, \lambda_i x_{j(i)},\ldots \lambda_m x_{j(m)},z,t).$$
Since $\alpha(\underline{x})=\underline{x}$, and $\alpha(z^q+t^r+x_0p(\underline{x}))=z^q+t^r+\nu_0x_0p(\underline{x})$, the $k$-algebra $\tilde{R}_1:=R_1\otimes_{\sigma_1(D)} \alpha(D)$ is isomorphic over $D$ to $R_m(\underline{n},\tilde{p}_1)$, where  $\tilde{p}_1=\nu_0 p_1$, and $\phi$ factors through a homomorphism of $k[x_1,\ldots,x_m]$-algebras $\tilde{\phi}:\tilde{R}_1\to R_2$. Replacing $R_1$ by $\tilde{R}_1$ and $\phi$ by $\tilde{\phi}$, the corresponding vertical maps in the diagram above are now isomorphisms of $k[x_1,\ldots, x_m]$-algebras. 
 Letting $\mathfrak{m}\subset k[x_1,\ldots,x_m]$ be the maximal ideal generated by the elements $x_i-1$, $i=1,\ldots$ and $\kappa(\mathfrak{m})=k[x_1,\ldots,x_m]/\mathfrak{m}\cong k$, $\tilde{\phi}$ induces an  isomorphism of $k$-algebras $\bar{\phi}$ between the $k$-algebras $$\tilde{R}_1\otimes_{k[x_1,\ldots,x_m]} \kappa(\mathfrak{m})\cong R_0(n_0,\nu_0p_1) 
 \;\textrm{and} \; 
R_2\otimes_{k[x_1,\ldots,x_m]}\kappa(\mathfrak{m}) \cong R_0(n_0,p_2)).$$
Since by assumption, $\deg p_1,\deg p_2\leq n_0-2$, it follows from \cite[Theorem 1.3]{DMP} that there exists $\lambda,\mu'\in k^*$ such that $p_2(u)=\mu' (\nu_0p_1)(\lambda u)$, whence that $p_2(u)=\mu p_1(\lambda u)$ for some $\lambda,\mu\in k^*$. 

Conversely, if $p_2(u)=\mu p_1(\lambda u)$ then the $k$-automorphism $\eta$ of $k[x_0,\ldots, x_m, y,z,t]$ defined by $$\eta(x_0,x_1,x_2\ldots,x_m,y,z,t)=(\mu x_0,\mu^{-1}\lambda x_1,x_2,\ldots,x_m,\mu^{n_1-n_0}\lambda^{-n_1}y,z,t)$$ induces an isomorphism $\phi:R_m(\underline{n},p_1)\to R_m(\underline{n},p_2)$.
\end{proof}

\subsection{Isomorphism types of $\mathbb{A}^1$-cylinders}

The following proposition shows that the $\mathbb{A}^1$-cylinders $X_m(\underline{n},p)\times_k \mathbb{A}^1_k$ over the $k$-varieties $X_m(\underline{n},p)$ are all isomorphic. Together with Proposition \ref{c}, it implies in particular that for every $m\geq 0$, the set of varieties $\{X_m(\underline{n},p),\, p \in k[u]\setminus uk[u]\}$ contains
an infinite collection of pairwise non-isomorphic  $k$-varieties $X_m(\underline{n},p_1)$ and $X_m(\underline{n},p_2)$ providing counterexamples to the generalized Cancellation Problem in dimension $m+3$.

\begin{prop}\thlabel{d} The $k$-algebras $R_m(\underline{n},p)^{[1]}$, $p\in k[u]\setminus uk[u]$, are all isomorphic. 
\end{prop}
\begin{proof}
Since every $k$-algebras $R_m(\underline{n},p)$ with $p(0)\neq 0$ is isomorphic to one $R_m(\underline{n},\tilde{p})$ with $\tilde{p}(0)=1$, we can henceforth assume without loss of generality that $p(0)=1$. With the notation of the proof of Proposition \ref{c}, the $k$-algebra $R_m(\underline{n},p)[w]$, where $w$ is a new variable, identifies with the affine blow-up algebra $$D[w][\frac{J_m(\underline{n},p)}{\underline{x}^{\underline{n}}}] \subset D[\underline{x}^{-1}][w].$$
By the universal property of affine blow-up algebras 
\cite[Proposition 2.1]{KZ}, $R_m(\underline{n},p)[w]$ and $R_m(\underline{n},1)[w]$ are isomorphic provided that there exists a $k$-automorphism $\psi$ of $D[w]$ mapping the ideal generated by $\underline{x}^{\underline{n}}$ onto itself and such that $\psi(J_m(\underline{n},p))=J_m(\underline{n},1)$. Given a polynomial $f\in k^{[1]}$ such that $\exp(\underline{x}f(\underline{x}))\equiv p(\underline{x})$ modulo $\underline{x}^{\underline{n}}$, similar arguments as in the proof of \cite[Theorem 1.3]{DMP} and \cite[Theorem 4.1]{DF} imply the existence of a pair of polynomials $g_1,g_2\in k^{[1]}$ such that 
$$ g_1(\underline{x}) \equiv \exp\left(\frac{1}{q}\underline{x} f(\underline{x})\right)  \; \textrm{ and }  g_2(\underline{x}) \equiv\exp\left(\frac{1}{r}\underline{x} f(\underline{x})\right)$$ modulo $\underline{x}^{\underline{n}}$ and such that  
 $\underline{x}^{\underline{n}}g_1(\underline{x}), \underline{x}^{\underline{n}}g_2(\underline{x})$ and $ g_1(\underline{x})g_2(\underline{x})$ generate the unit ideal in $k[x_0,\ldots,x_m]$. It follows that there exist polynomials $h_1,h_2,h_3 \in k[x_0,\ldots,x_m]$ such that the matrix 
$$
\begin{pmatrix}
g_1(\underline{x}) &0 &\underline{x}^{\underline{n}}\\
0 &g_2(\underline{x}) &\underline{x}^{\underline{n}}\\
h_1 &h_2 &h_3
\end{pmatrix} \in \mathrm{Mat}_{3,3}(k[x_0,\ldots,x_m])
$$
    belongs to $\mathrm{GL}_3(k[x_0,\ldots,x_m])$.  
    Letting $\psi$ be the   $k[x_0,\ldots, x_m]$-algebra automorphism of the polynomial ring $D[w]=k[x_0,\ldots, x_m][z,t,w]$ defined by the above matrix, we have
    $$   \psi(z^q+t^r+x_0p(\underline{x}))=(g_1(\underline{x})z+\underline{x}^{\underline{n}}w)^q+(g_2(\underline{x})t+\underline{x}^{\underline{n}}w)^r+x_0p(\underline{x})= p(\underline{x})(z^q+t^r+x_0) +\underline{x}^{\underline{n}}F,
    $$
    for some polynomial $F\in D[w]$. Hence, $\psi$ maps the ideal $J_m(\underline{n},p)=(\underline{x}^{\underline{n}},z^q+t^r+x_0p(\underline{x}))$ onto the ideal $(\underline{x}^{\underline{n}},p(\underline{x})(z^q+t^r+x_0))$. Since $p(\underline{x})$ is invertible modulo $\underline{x}^{\underline{n}}$, the latter ideal is equal to $J_m(\underline{n},1)$ and the assertion follows. 
\end{proof}

We also record the following result, which  implies in particular that the only algebraic invariants known so far which succeeded to distinguish the varieties $X_m(\underline{n},p)$ from affine spaces do not allow to decide whether the varieties $X_m(\underline{n},p)\times _k \mathbb{A}_k^1$ are affine spaces or not. 

\begin{prop}\thlabel{f} For every triple $(m,\underline{n},p)$, the Makar-Limanov and Derksen invariants of the $k$-algebra $R_m(\underline{n},p)^{[1]}$ are respectively equal to the subalgebras
$$\mathrm{ML}(R_m(\underline{n},p)^{[1]})=k \quad \textrm{and} \quad  \mathrm{DK}(R_m(\underline{n},p)^{[1]})=R_m(\underline{n},p)^{[1]}.$$ 
\end{prop}

\begin{proof}
In view of Proposition \ref{d}, we can assume without loss of generality that $p=1$. Put  $R=R_m(\underline{n},1)$ and $R[w]=R_m(\underline{n},p)^{[1]}$. The equality $\mathrm{DK}(R[w])=R[w]$ follows from the fact  that every nonzero locally nilpotent $k$-derivation of $R$ extends to a locally nilpotent $k[w]$-derivation of $R[w]$ and that $R\subset \mathrm{Ker}(\tfrac{\partial}{\partial w})$. 
On the other hand, it follows from Proposition \ref{b}, that $\mathrm{ML}(R[w]) \subset \mathrm{ML}(R)=  k[x_{0},\ldots,x_m]$. To complete the proof,  we now argue that for every $i\in\{0,1,\ldots,m\}$ there exists
a locally nilpotent $k$-derivation $\tilde{\partial}_i$ of $R[w]$ which
does not contain $x_{i}$ in its kernel. 
Letting $K=k(x_1,\ldots,x_m)$, we have $$R\otimes_{k[x_1,\ldots,x_m]} K\cong K[x_0,Y,z,t]/(x_0^{n_0}Y+z^q+t^r+x_0) \cong R_0(n_0,1)\otimes_k K,$$ where $Y=(\prod_{i=1}^mx_i^{n_i}) y$. Since $\mathrm{ML}(R_0(n_0,1)[w])=k$ by \cite{D}, there
exists a locally nilpotent $K$-derivation $\partial$
of $R[w]\otimes_{k[x_0,\ldots, x_m]} K$ which does not have $x_{0}$ in its kernel. Since $R[w]\otimes_{k[x_0,\ldots, x_m]} K$ is a finitely generated $K$-algebra,
it follows in turn that for suitably chosen $f\in k[x_1,\ldots,x_m]$, $ f\partial$ restricts to a nonzero locally nilpotent $k[x_1,\ldots,x_m]$-derivation $\tilde{\partial}_0$
of $R[w]\subset R[w]\otimes_{k[x_0,\ldots, x_m]} K$ such that $x_0 \notin \mathrm{Ker} \, \tilde{\partial}$. Next, for $i\in \{1,\ldots,m\}$, let $L_i=k(x_0,x_{1},\ldots,x_{i-1},x_{i+1}\ldots,x_{m},t)$. Then $$ R\otimes_{k[x_0,x_1,\ldots,x_{i-1},x_{i+1}\ldots,x_{m},t]} L_i \cong L_i[x_i,Y,z]/(x_{i}^{n_{i}}Y+z^{q}+t^{r}+x_{0}),$$
where $Y=x_{0}^{n_{0}}\prod_{j\neq i}x_{j}^{n_{j}}y$. The latter $L_i$-algebra is the coordinate ring $B$ of a smooth Danielewski surface in $\mathbb{A}^3_{L_i}$ for which it can be verified by similar arguments as in \cite{D} that $\mathrm{ML}(B[w])=L_i$. The same reasoning as in the previous case then gives the existence of a locally nilpotent $k[x_0,x_1,\ldots,x_{i-1},x_{i+1}\ldots,x_{m},t]$-derivation $\tilde{\partial}_i$ of $R[w]$ such that $x_i \notin \mathrm{Ker} \, \tilde{\partial}_i$.
\end{proof}

\section{$\mathbb{A}^1$-contractibility}

Recall that by Proposition \ref{b}, the $k$-varieties $X_m(\underline{n},p)=\mathrm{Spec}(R_m(\underline{n},p))$ are all non-isomorphic to $\mathbb{A}^{m+3}_k$. The next proposition implies in turn that all these are "exotic" $\mathbb{A}^1$-contractible affine $k$-varieties. 

\begin{prop}\thlabel{a} Every affine $k$-variety $X_m(\underline{n},p)$ as in Notation \ref{mainnota} is $\mathbb{A}^{1}$-contractible. 
\end{prop}

\begin{proof}
Since $X_m(\underline{n},p)\times_k \mathbb{A}^1_k\cong X_m(\underline{n},1)\times_k\mathbb{A}^1_k$ by Proposition \ref{d}, it suffices to show, by $\mathbb{A}^1$-invariance, that $X_m:=X_m(\underline{n},1)$ is $\mathbb{A}^1$-contractible. The result for $m=0$ being  established in \cite{DF}, we henceforth  assume that $m>0$. We put $A_{m-1}=k[x_1,\ldots, x_{m-1},z,t]$, $R_m=A_{m-1}[x_0,x_m,y]/(\prod_{i=0}^m x_i^{n_i}y+z^q+t^r+x_0)$ and we define the following closed subschemes of $X_m$:
\begin{align*}
W_m=\{x_m=0\} & \cong  \mathrm{Spec}(R_m/x_mR_m)  \cong  \mathrm{Spec}(A_{m-1}[y])  \cong \mathbb{A}_k^{m+2} &\\
H_{m}=\{y=0\} & \cong  \mathrm{Spec}(R_m/yR_m) \cong  \mathrm{Spec}(A_{m-1}[x_m]) \cong  \mathbb{A}_k^{m+2}&\\
P_m=W\cap H_m & \cong  \mathrm{Spec}(R_m/(x_m,y)R_m)  \cong  \mathrm{Spec}(A_{m-1})  \cong  \mathbb{A}_k^{m+1}. &
\end{align*}
Denote by $j_m:H_m\to X_m$  and $i_m: H_m\setminus P_m \to  X_m\setminus W_m$ the natural closed immersions and observe that for $Y=x_my$, we have an isomorphism $$X_m\setminus W_m\cong \mathrm{Spec}(A_{m-2}[x_0,x_{m-1},Y][x_m^{-1}]/(x_0^{n_0}\prod_{i\neq m}x_{i}^{n_{i}}Y+z^{q}+t^{r}+x_{0}))\cong X_{m-1}\times_k\mathbb{G}_{m,k}$$ 
for which $i_m$ equals the product $$j_{m-1}\times \mathrm{id}_{\mathbb{G}_{m,k}}: H_m\setminus P_m \cong H_{m-1}\times_k \mathbb{G}_{m,k}= \mathrm{Spec}(A_{m-2}[x_{m-1}][x_m^{\pm 1}]))\to X_{m-1}\times_k\mathbb{G}_{m,k}.$$
Now consider the following diagram of cofiber sequences in  $\mathbf{H(k)}$
$$
\begin{tikzcd}
H_{m}\setminus P_{m} \arrow[r] \arrow[d, "i_{m}"]  & H_{m}  \arrow[r] \arrow[d, "j_{m}"] & H_{m}/(H_{m}\setminus P_{m}) \arrow[d, "l_{m}"]\\
X_{m}\setminus W_m \arrow[r] & X_{m} \arrow[r] & X_{m}/(X_{m}\setminus W_m)
\end{tikzcd}
$$
associated to the open immersions $H_m\setminus P_m \to H_m$ and $X_m\setminus W\to W_m$. Since $H_m$ is $\mathbb{A}^1$-contractible, $X_m$ is $\mathbb{A}^1$-contractible provided that $j_m$ is an $\mathbb{A}^1$-weak equivalence. The case $m=0$ follows from the fact that $j_0: H_0 \to X_0$ is an $\mathbb{A}^1$-weak equivalence (cf. \cite{DF}) and we now establish this property for $m>0$ by induction on $m$. 

Since $H_m$ is $\mathbb{A}^1$-contractible, it is enough, by the weak five lemma \cite[Lemma 2.1]{DF}, to verify that $i_m$ and $l_m$ are both $\mathbb{A}^1$-weak equivalences. The property for $i_m$ follows from the observation made above that $i_m=j_{m-1}\times \mathrm{id}_{\mathbb{G}_{m,k}}$ and the assumption hypothesis that $j_{m-1}$ is an $\mathbb{A}^1$-weak equivalence. On the other hand, since $P_{m}$ is the complete intersection of $W$
and $H_{m}$ in $X_{m}$, the normal bundle $N_{P_{m}/H_{m}}$ of $P_m$ in $H_m$ equals the restriction to $P_m$ of the normal bundle $N_{W_m/X_m}$ of $W_m$ in $X_m$, which is trivial. Since $X_m$, $W_m$, $P_m$ and $H_m$ are smooth schemes, it follows from homotopy purity (\cite[Section 3.2, Theorem 2.23]{MV}) that $H_{m}/(H_{m}\setminus P_{m})\simeq \mathrm{Th}(N_{P_{m}/H_{m}})\simeq P_{m,+}\wedge\mathbb{P}^{1}$
and that $X_{m}/(X_{m}\setminus W_m)\simeq \mathrm{Th}(N_{W_m/X_m}) \simeq W_{m,+}\wedge\mathbb{P}^{1}$. 
Under these isomorphisms the map $l_m$ coincides with map 
\[
\mathrm{Th}(N_{P_{m}/H_{m}})\simeq P_{{m},+}\wedge\mathbb{P}^{1}\to W_{m,+}\wedge\mathbb{P}^{1}\simeq \mathrm{Th}(N_{W_m/X_m})
\]
obtained as the $\mathbb{P}^{1}$-suspension
of the closed immersion $P_{m}\hookrightarrow W_m$. The latter being an $\mathbb{A}^{1}$-weak equivalence, we conclude that $l_m$ is an $\mathbb{A}^{1}$-weak equivalence as well. 
\end{proof}

Proposition \ref{f} and Proposition \ref{a} imply that the $k$-varieties $X_m(\underline{n},p)\times _k \mathbb{A}_k^1 \cong X_m(\underline{n},1)\times _k \mathbb{A}_k^1$
are all $\mathbb{A}^1$-contractible and undistinguishable from an affine space of dimension $m+4$ by means of known algebro-geometric or $\mathbb{A}^1$-homotopic invariants. All of them are thus potential counterexamples to the Zariski Cancellation Problem in dimension $m+3$.

\medskip

We conclude with the construction of an example of an algebraic family of varieties of the form $X_m(\underline{n},p)$ parametrized by an affine space as announced in the introduction:

\begin{ex} \label{exa:part1} Let $k_0$ be a field of characteristic zero, let $m\geq 0$, $\underline{n}=(n,\ldots, n)$ for some integer $n\geq 4$, let  $A=k_0[a_2,\ldots, a_{n-2}]$, $P=1+u+\sum_{j=2 }^{n-2} a_ju^j \in A[u]$ and consider the $A$-algebra $$\mathcal{R}_m(\underline{n},P):=A[x_0,\ldots, x_m,y,t,z]/(\underline{x}^{\underline{n}}y+z^r+t^q+x_0P(\underline{x})).$$
It follows from Proposition \ref{b}  and Proposition \ref{a} that the smooth affine morphism 
$$\pi:\mathcal{X}_m(\underline{n},P):=\mathrm{Spec}(\mathcal{R}_m(\underline{n},P))\to S:=\mathrm{Spec}(A)\cong \mathbb{A}_{k_@}^{n-3}$$ can be viewed as an algebraic family 
whose fibers over points of $S$, closed or not, are all $\mathbb{A}^1$-contractible varieties of dimension $m+3$ non-isomorphic to affine spaces over the corresponding residue fields. The fibers of $\pi$ over the $k_0$-rational points of $S$ are pairwise non-isomorphic by Proposition \ref{c} whereas their $\mathbb{A}^1$-cylinders are all isomorphic by Proposition \ref{d}. 

The construction in the proof of Proposition $\ref{d}$ even provides a nonzero element $r\in A$ such that the $A_r$-algebras 
the $A$-algebras 
$\mathcal{R}_m(\underline{n},P)^{[1]}\otimes_A A_r$ and $(R_m(\underline{n},1)^{[1]})\otimes_k A_r$ are isomorphic. Namely, writing $$\ln (P)=\ln (1+u+\sum_{j=2}^{n-2}a_ju^j)=\sum_{i=1}^{n-1} \gamma_iu^i +u^n H=uF(u)+u^n S\in k_0[[a_2,\ldots, a_{n-2},u]] $$ for some uniquely determined elements $\gamma_i\in A$, $F \in A[u]$ and $S\in k_0[[a_2,\ldots, a_{n-2},u]]$, we have $ \exp(uF(u))=P(u)+u^nT$,  and we can write 
$$\exp(\frac{1}{q} uF(u))=G_1(u)+u^n T_1 \; \textrm{and} \; \exp(\frac{1}{r}uF(u))=G_2(u)+u^nT_2$$ for some $T,T_1,T_2\in k_0[[a_2,\ldots, a_{n-2},u]]$
and polynomials $G_1,G_2\in A[u]$  which can be chosen to be monic and so that their images in $\mathrm{Frac}(A)[u]$ generate the unit ideal. This implies that the resultant $r=\mathrm{Res}(G_1,G_2)\in A $ of $G_1$ and $G_2$ with respect to the variable $u$ is non-zero, whence that the images of $G_1$ and $G_2$ in $A_r[u]$ are relatively prime. It follows in turn that $u^nG_1(u)$,$u^nG_2(u)$ and  $G_1(u)G_2(u)$ generate the unit ideal of $A_r[u]$. Then, as in the proof of Proposition \ref{d}, there exists  $H_1,H_2,H_3\in A_r[x_0,\ldots,x_m]$ such that 
the $A_r[x_0,\ldots, x_m]$-algebra automorphism of $A_r[x_0,\ldots, x_m][z,t,w]$ associated the matrix 
$$
\begin{pmatrix}
G_1(\underline{x}) &0 &\underline{x}^{\underline{n}}\\
0 &G_2(\underline{x}) &\underline{x}^{\underline{n}}\\
H_1 & H_2 &H_3
\end{pmatrix} \in \mathrm{GL}_3(A_r[x_0,\ldots, x_m])
$$
maps the ideal $(\underline{x}^{\underline{n}},z^r+t^q+x_0P(\underline{x})$ isomorphically onto the one $(\underline{x}^{\underline{n}},z^q+t^r+x_0)$ whence lifts to an $A_r[x_0,\ldots,x_m]$-algebra isomorphism between $\mathcal{R}_m(\underline{n},P)\otimes_A A_r[w]$ and $(R_m(\underline{n},1)[w])\otimes_k A_r$. 

Letting $S_r=\mathrm{Spec}(A_r)$ be the corresponding principal Zariski open subset of $S$, the above isomorphism translates geometrically into the property that the restricted family $$\pi|_{S_r}:\mathcal{X}_m(\underline{n},P)|_{S_r}=\mathcal{X}_m(\underline{n},P)\times_S S_r\to S_r$$ becomes isomorphic after taking product over $S_r$ with $\mathbb{A}^1_{S_r}$ to the trivial product family
$$\mathrm{pr}_1 : S_r\times_k (\mathrm{Spec}(R_m(\underline{n},1))\times_k \mathbb{A}^1_k)\to S_r.$$ 
It follows in particular that $\pi|_{S_r}:\mathcal{X}_m(\underline{n},P)|_{S_r}\to S_r$ is an $\mathbb{A}^1$-weak equivalence. 
\end{ex}

\bibliography{nontriv_A1_v3-short}
\bibliographystyle{amsplain}
\end{document}